\documentclass[12pt,twoside,leqno]{amsart}
\usepackage{amsmath,amssymb,
amsbsy,amsfonts,latexsym,amsopn,
amstext,cite,amsxtra,euscript,amscd,bm,mathabx}
\usepackage{url}

\usepackage[colorlinks,linkcolor=blue,
anchorcolor=blue,citecolor=blue,backref=page]{hyperref}
\usepackage{color}
\usepackage{graphics,epsfig}
\usepackage{graphicx}
\usepackage{float}
\usepackage{epstopdf}
\hypersetup{breaklinks=true}

\usepackage[norefs,nocites]{refcheck}

\newtheorem{theorem}{Theorem}[section]
\newtheorem{lemma}[theorem]{Lemma}

\newtheorem{corollary}[theorem]{Corollary}

\numberwithin{equation}{section}
\numberwithin{theorem}{section}

\def\e{{\mathbf{\,e}}}

\newcommand{\floor}[1]{\lfloor #1 \rfloor}

\newcommand{\R}{\mathbb{R}}

\newcommand{\N}{\mathbb{N}}
\newcommand{\Z}{\mathbb{Z}}
\newcommand{\T}{\mathbb{T}}

\newcommand{\PP}{\mathbb{P}}
\newcommand{\EE}{\mathbb{E}}

\newcommand{\G}{\mathcal{G}}

\newcommand{\I}{\mathcal{I}}

\def\vy{\mathbf y}

\def\vx{\mathbf x}

\def\vu{\mathbf u}
\def\vv{\mathbf v}
\def\vw{\mathbf w}

\title{On uniform distribution of  $\alpha\beta$-orbits}

\author[C. Chen]{Changhao Chen}

\address{Department of Pure Mathematics, University of New South Wales,
Sydney, NSW 2052, Australia}
\email{changhao.chenm@gmail.com}

\author[X. Wang]{Xiaohua Wang}

\address{China Economics and Management Academy, Central University of Finance and Economics, Beijing 100081, China}
\email{fridman\_wxh@163.com}

\author[S. Wen]{Shengyou Wen}

\address{Department of Mathematics and Key Laboratory of Applied Mathematics,  Hubei University, Wuhan 430062, China}
\email{sywen\_65@163.com }


\begin{document}

\begin{abstract} 
Let $\alpha, \beta \in (0,1)$ such that at least one of them is irrational.  We take a random walk on the real line such that  the choice of $\alpha$ and $\beta$ has  equal probability $1/2$. We prove that almost surely the $\alpha\beta$-orbit is uniformly distributed module one, and the exponential sums along its  orbit has the square root cancellation. We also show that the exceptional set in the probability space, which does not have the property of uniform distribution modulo one, is large in the terms of topology and  Hausdorff dimension.
\end{abstract}

\maketitle

\section{Introductions}

\subsection{Motivation and previous results}
 Let  $\alpha, \beta \in \T$ where $\T$ is the circle group $\R/\Z$. A non-empty closed set $K\subseteq \T$  is called an $\alpha\beta$-set if 
\begin{equation*}
K\subset (K-\alpha)\cup (K-\beta),
\end{equation*}
or equivalently, if $x\in K$ then $x+\alpha\in K$ or $x+\beta\in K$.  Moreover a sequence $\xi_n, n\in \N$ is called an $\alpha\beta$-orbit if 
\begin{equation*}
\xi_{n+1}-\xi_n\in \{\alpha, \beta\},
\end{equation*}
where  $\xi_0=0$. We remark that  the above sums and differences are taking in the space $\T$, and thus the values are  always taken   modulo  one.  

 We first show  some known results in the following.  Engelking \cite{E} and Katznelson \cite{K} initially studied the $\alpha\beta$-sets and $\alpha\beta$-orbits. Katznelson \cite{K} proved that if   $\alpha, \beta\in \T$ are rational independent ($\text{mod}\,1$), then there exists  a non-where dense $\alpha\beta$-set; furthermore for some $\alpha, \beta$, there exits $\alpha\beta$-set with Hausdorff dimension zero.   For Hausdorff dimension (see Subsection~\ref{subsec:H} below) and the box counting dimension, we refer to Falconer \cite{Falconer},  Mattila \cite{Mattila1995}.

Recently Feng and Xiong \cite{FX}, motivated from the problem of ``affine embeddings of Cantor sets",  proved that if $\alpha, \beta\in \T$  are rational independent ($\text{mod}\,1$), then for any $\alpha\beta$-set $K$,  either $K-K=\T$ or $K$ has non-empty interior. It follows that 
\begin{equation*}
\underline{\dim}_B K\ge 1/2.
\end{equation*}
We remark that Feng and Xiong \cite{FX} also considered the general case which contains more parameters than just $\alpha, \beta$, and they made applications to ``affine embeddings of Cantor sets".

Yu \cite{Yu} showed some recurrent conditions such that  the $\alpha\beta$-orbit is dense $\T$.  Moreover Yu \cite{Yu} also considered the general case (more parameters) and improved  a result of Feng and Xiong \cite{FX} for the case when the parameters  are rationally independent algebraic numbers.

In this project we mainly consider the $\alpha\beta$-orbits. We remark that the $\alpha\beta$-sets are closely related to the $\alpha\beta$-orbits. For instance it is not hard to see that 
any $\alpha\beta$-orbit is an $\alpha\beta$-set.

\subsection{Set-up and main results}

We first show some notation for our purpose. Throughout the paper, let $\alpha, \beta\in (0,1)$ such that  at least one of them  is irrational.  Let  
$$
\Sigma=\{(x_1, x_2, \ldots, x_n, \ldots ): x_n \in \{1, 2\}\}.
$$
For $\vx=(x_1, x_2, \ldots )\in \Sigma$ and each $n\in \N$ define
$$
S_n(\vx)=\sum_{i=1}^{n} X_i(\vx) 
$$
where
\begin{equation*}
X_i(\vx) =
  \begin{cases}
  \alpha  & x_i=1, \\
  \beta & x_i=2.
  \end{cases}
\end{equation*}

It is convenient to consider $\Sigma$ as   probability space, and $X_i$ with $i \in \N$ is a sequence of random variables.   For each $i$  the random variable $X_i$ takes value $\alpha$ or $\beta$ with the equal probability $1/2$. This gives us a probability measure on $\Sigma$, and denote it  by $\PP$.

We will show that for $\PP$ almost all $\vx\in \Sigma$ the sequence  $S_n(\vx), n\in \N$ is uniform distributed module one.  
Recalling that  a sequence of real numbers  $\xi_n, n \in \N$ is called uniformly distributed module one (u.d. mod $1$) if for any interval $I\subseteq [0,1)$ one has 
\begin{equation*}
\frac{\# \{1\le n\le N: \{\xi_n\}\in I\}}{N}\rightarrow |I| \quad \text{as} \quad N \rightarrow \infty.
\end{equation*}
Here $\{x\}$ is the fractional part of $x$, the notation $\# S$ means the cardinality of the set $S$, and $|I|$ stands for the length of the interval $I$.  See Drmota and Tichy \cite{DT} and Harman \cite[Chapter 5]{Harman} for more details.

We remark that if $\alpha=\beta$ and $\alpha$ is irrational, then it is known over a century that  the sequence $n\alpha, n\in \N$, also known as irrational rotation in dynamical system, is uniformly distributed modulo one. We refer to \cite[Chapter 1]{DT} for more details.

Clearly, if both of $\alpha$ and $\beta$ are rational, then the $\alpha\beta$-orbits (modulo one) are only finite points in $\mathbb{T}$, and hence never being u.d. mod $1$. Excluding this trivial case  we obtain the following.

\begin{theorem}
\label{thm:uniform} For $\PP$ almost all $\vx\in \Sigma$ the sequence  $S_n(\vx), n\in \N$ is uniform distributed  module one.
\end{theorem}

We note that in order to prove Theorem~\ref{thm:uniform}, we  show that for almost all $\vx\in \Sigma$ and any non-zero $h\in \Z$ one has 
\begin{equation}
\label{eq:zero}
\lim_{N\rightarrow \infty}\frac{1}{N}\sum_{n=1}^{N}\e(hS_n(\vx))=0,
\end{equation}
where  throughout the paper
we denote $\e(x)=\exp(2 \pi ix)$.
Then by using Weyl criterion (see below Lemma~\ref{lem:Weyl}) we obtain the claim of Theorem~\ref{thm:uniform}.  Indeed by adapting  the  argument of the first author and Shparlinski~\cite{ChSh}, for any fix $h\in \Z\setminus\{0\}$ we improve the order of the convergence of~\eqref{eq:zero}.

\begin{theorem}
\label{thm:exponentialsums} Let $h\in \Z\setminus\{0\}$. Then 
for $\PP$ almost all $\vx\in \Sigma$ and  for any $\varepsilon>0$ we have  
$$
\sum_{n=1}^N \e(hS_n( \vx))=O_{\alpha, \beta, \varepsilon}(N^{1/2} (\log N)^{3/2+\varepsilon}).
$$
\end{theorem}
We conjecture  that  this upper bound is the best possible except for a logarithm factor. Moreover we do not know how to deal with the following question. Is it true that for almost all $\vx\in \Sigma$ one has that for any $h\in \Z\setminus\{0\}$ and any $\varepsilon>0$,
$$
\sum_{n=1}^N \e(hS_n( \vx))=O_{\alpha, \beta, \varepsilon}(N^{1/2} (\log N)^{3/2+\varepsilon})\,?
$$

We now turn to the study  of the size of the exceptional set. To be precise let
$$
E_{\alpha, \beta}=\{\vx\in \Sigma: S_n(\vx), n\in \N, \text{ is not u.d. mod 1}\}.
$$

We give the ordinary metric on the space $\Sigma$ in the following.
Let  $\vx=(x_1,\ldots), \vy=(y_1, \ldots)\in \Sigma$. If  $x_1=y_1$ then  define  
\begin{equation*}
d(\vx, \vy)=2^{-k(\vx, \vy)},
\end{equation*}
where
\begin{equation*}
k(\vx, \vy)=\max\{n: x_j=y_j, j\le n\}.
\end{equation*}
On the other hand  if $x_1\neq y_1$ then let $d(\vx, \vy)=1$. Specially let $d(\vx, \vx)=0$ for any $\vx\in \Sigma$. 
Clearly the function $d(\vx, \vy)$ is a metric on the space $\Sigma$.

Recalling that a $\G_{\delta}$  set is a subset of a topological space that is a countable intersection of open sets.

\begin{theorem}
\label{thm:topology} The set $E_{\alpha, \beta}$ contains a dense $\G_\delta$ set in $\Sigma$.
\end{theorem}

Theorem \ref{thm:uniform} implies that the set $E_{\alpha, \beta}$ has zero measure. Therefore it is natural to ask what is the Hausdorff dimension of the set $E_{\alpha, \beta}$. We note that the metric space $(\Sigma, d)$ has Hausdorff dimension one. The following result shows that the exceptional set $E_{\alpha, \beta}$ is large in the sense of Hausdorff dimension.

\begin{theorem}
\label{thm:dimension} The set $ E_{\alpha, \beta}$ has Hausdorff dimension one.
\end{theorem}

\section{Preliminaries}

\subsection{Notation and conventions.}

As usual, the notations $U = O(V )$, 
$U \ll V$ and $ V\gg U$  are equivalent to $|U|\leqslant c|V| $ for some positive constant $c$. The notation $U_n=o(V_n)$ means that for any $\varepsilon>0$ one has $|U_n|\le \varepsilon |V_n|$ for all large enough $n$.
Throughout the paper, all implied constants are absolute.

\subsection{Symbolic space}
Let $\Sigma=\{1, 2\}^{\N}$ be the collection of all the infinite words over the alphabet set $\{1, 2\}$. For each $n\in \N$ and $\vx=x_1x_2\ldots \in \Sigma$  let $\vx|_n=x_1\ldots x_n$.  Let  $\Sigma_n=\{\vx|_n: \vx \in \Sigma\}$ for each $n\in \N$ and $\Sigma_{*}=\bigcup_{n=1}^{\infty} \Sigma_n$ be the collection of all the words with finite length. For any $\vx\in \Sigma \cup \Sigma_*$ the length of $\vx$ is denote d by $|\vx|$. The combination of two words $\vx\in \Sigma_*$ and $\vy\in \Sigma_* \cup \Sigma$ is denoted by $\vx\vy$. Note that we do not have the word $\vy\vx$ if $\vy\in \Sigma$.    For $\vw\in \Sigma^{*}$ denote 
$$
[\vw]=\{\vx\in \Sigma: \vx|_{|\vw|}=\vw\},
$$
and call it a cylinder set. 


\subsection{Hausdorff dimension}
\label{subsec:H}

 The  Hausdorff dimension of a set $A\subseteq \R^{d}$ is defined as 
\begin{align*}
\dim_H A=\inf \{s>0: &~\forall \,   \varepsilon>0,~\exists \, \{ U_i \}_{i=1}^{\infty}, \ U_i \subseteq \R^{d},\  \text{such that }  \\
  &  A\subseteq \bigcup_{i=1}^{\infty} U_i \text{ and } \sum_{i=1}^{\infty} \text{diam} (U_i)^{s}<\varepsilon \}.
\end{align*}

A useful way to obtain the lower bound of Hausdorff dimension is by using 
the following  {\it mass distribution
 principle\/}~\cite[Theorem~4.2]{Falconer}.

\begin{lemma} 
\label{lem:mass}
Let $X$ be a metric space.  Let $A\subset X$ and $\mu$ be a measure on $X$ such that $\mu(A)>0$. If for any ball $B(x, r)$ with $0<r\leqslant r_0$ for some $r_0>0$ we have 
$$
\mu(B(x, r))\ll r^{s},
$$
then the Hausdorff dimension of $A$ is at least $s$.
\end{lemma}

\subsection{A class of Cantor sets}
\label{sec:Cantor}

For proving Theorem \ref{thm:dimension} we construct a Cantor set which is a subset of $E_{\alpha \beta}$, and then apply Lemma \ref{lem:mass} to obtain the lower bound of the dimension of this Cantor set. 
 
We first show some parameters.  
Fix $0<\varepsilon<1$. Let $n_k$ be a sequence positive integers  such that $n_{k+1}>2n_{k}$ for each $k\in \N$. Denote 
\begin{equation}
\label{eq:nk}
\widetilde{n}_k=n_k+\floor{\varepsilon n_k}.
\end{equation}
Let $q_1=2^{n_1}$ and for each $k\in \N$ denote 
$$
q_{k+1}=2^{n_{k+1}-\widetilde{n}_k}.
$$
 
Now we show the construction of a class of  Cantor set in $\Sigma$. The construction is motivated from Lemma \ref{lem:jumper2} below. 

For any $\vu\in \Sigma_n$ we pick $\vv(\vu)\in \Sigma_{\floor{\varepsilon n}}$ and denote 
\begin{equation*}
B_n=\bigcup_{\vu\in \Sigma_n}[\vu \vv(\vu)].
\end{equation*}
 
We remark that for each $\vu$ we have several different choices, but this will not affect our result.

For each $k\in \N$ let 
\begin{equation*}
E_k=\bigcap_{j=1}^{k} B_{n_j},
\end{equation*}
and 
\begin{equation*}
E=\bigcap_{k=1}^{\infty} E_k.
\end{equation*}
 
Note that $E_k\supseteq E_{k+1}$ and 
\begin{equation*}
\# E_k:=\#\{\vw\in \Sigma_{\widetilde{ n}_k}: [\vw] \subset E_k\}=\prod_{j=1}^{k}q_j.
\end{equation*}

\begin{lemma} 
Using above notation, we have 
\begin{equation*}
\dim_H E =\liminf_{k\rightarrow \infty} \frac{\log \prod_{j=1}^{k}q_j}{\widetilde{ n}_k\log 2}.
\end{equation*}
\end{lemma}

\begin{proof}
For convenience of notation, let
\begin{equation*}
s=\liminf_{k\rightarrow \infty} \frac{\log \prod_{j=1}^{k}q_j}{\widetilde{ n}_k\log 2}.
\end{equation*}
Note that $0\le s\le 1$. This follows by~\eqref{eq:nk}, the definition of $\widetilde{ n}_k$, and  
$$
 \prod_{j=1}^{k}q_j=2^{n_1}2^{n_2-\widetilde{ n}_1}\ldots 2^{n_k-\widetilde{ n}_{k-1}}\le 2^{n_k}.
$$

Let $0<\varepsilon<t-s$ for some $\varepsilon$ and $t$. By the definition of $s$,  there exists a subsequence of $\N$, for convenience of notation we assume that this subsequence is $\N$,  such that 
\begin{equation}
\label{eq:number}
\prod_{j=1}^k q_j \le 2^{(t-\varepsilon)\widetilde{ n}_k}
\end{equation}
for all large enough  $k$. 

Observe that  for each $k$ the cylinders of $E_k$ form a cover of $E$ and each of these cylinder has diameter $2^{-\widetilde{ n}_k}$. Combining with \eqref{eq:number} we have 
\begin{equation*}
\# E_k 2^{-t\widetilde{ n}_k} \le  2^{-\varepsilon \widetilde{ n}_k}.
\end{equation*}
Thus the definition of the Hausdorff dimension implies that $\dim_H E \le t$.  By the arbitrary choice of $t>s$ we obtain that $\dim_H E\le s$.

Now we turn to the lower bound of $\dim_H E$. We assume that $0<t<s$.  Thus for all large enough $k$ we have  
\begin{equation*}
\# E_k \ge 2^{t\widetilde{n}_k}.
\end{equation*}

We first construct a measure on  $E$.  For each $k$ let $\nu_k$ be a measure on $\Sigma$ such that $\nu_k(\Sigma)=1$ and for each word $\vw\in \Sigma_{\widetilde{n}_k}$ and $[\vw]\subset E_k$  one has 
$$
\nu_k([\vw])=\frac{1}{\# E_k}.
$$
The measure $\nu_k$ weakly convergence to a measure $\nu$, see
\cite[Chapter 1]{Mattila1995}.

For any $B(\vx, r)\subset \Sigma$ with  $0<r<1$ there exits $k$ such that 
$$
2^{-n_{k+1}}<r\le 2^{-n_k}.
$$
We  now do a case by case argument due to the range of $r$.

{\bf Case 1:}  Suppose that $2^{-\widetilde{n}_k}\le r\le 2^{-n_k}$. Since 
\begin{equation*}
\#\{\vw\in \Sigma_{n_k}: B(\vx, r)\cap [\vw]\neq \emptyset\}\le 2,
\end{equation*}
we obtain
\begin{equation*}
\nu(B(\vx, r)) \le \frac{2}{\# E_{n_k}}\ll 2^{-\widetilde{n}_k t}\ll r^{t}.
\end{equation*}

{\bf Case 2:} Suppose that $2^{-n_{k+1}}<r<2^{-\widetilde{n}_k}$. Since 
\begin{equation*}
\#\{\vw\in \Sigma_{ \widetilde{n}_{k+1}}: [\vw]\subset E_{k+1}, B(\vx, r)\cap [\vw]\neq \emptyset\}\ll \frac{r}{2^{-n_{k+1}}},
\end{equation*}
we obtain
\begin{equation*}
\nu(B(\vx, r))\ll \frac{r}{2^{-n_{k+1}}} \frac{1}{\# E_{k+1}}.
\end{equation*}
Recalling that $q_{k+1}=2^{n_{k+1}-\widetilde{ n}_k}$. Then (note that $t<1$)
\begin{equation*}
 \nu(B(\vx, r)) \ll \frac{r}{2^{-\widetilde{ n}_k}} \frac{1}{\# E_{k}}\ll r2^{\widetilde{ n}_k(1-t)}\ll r^{t}.
\end{equation*} 
Applying Lemma \ref{lem:mass} we obtain $\dim_H E\ge t$,  and this gives $\dim_H E\ge s$ by the arbitrary choice of $t<s$. Combining with the previous upper bound, $\dim_H E\le s$, we finish the proof.
\end{proof}

\begin{corollary}
\label{cor:dimension}
Using the above notation, let $n_k\nearrow\infty$ rapidly such that 
\begin{equation*}
\lim_{k\rightarrow \infty}\frac{\log \prod_{j=1}^k q_j}{\log q_{k+1}}=0,
\end{equation*}
Then $\dim_H E= 1/(1+\varepsilon)$.
\end{corollary}
 
\section{Proofs of main results}

\subsection{Proof of Theorem \ref{thm:uniform}}

We will use the following Weyl's criterion  \cite[Theorem 1.19]{DT}.


\begin{lemma}[Weyl criterion]
\label{lem:Weyl} A sequence $(x_n)_{n\ge 1}$ of points in $\R$ is u.d. mod $1$ if and only if 
\begin{equation*}
\lim_{N\rightarrow \infty}\frac{1}{N} \sum_{n=1}^{N}\e(hx_n)=0
\end{equation*}
holds for all non-zero integral point $h\in \Z\setminus \{0\}$.
\end{lemma}
 
We will also need the following second moment estimate.

\begin{lemma}
\label{lem:second-moment}
For any  $h\in \Z\setminus\{0\}$ we have 
$$
\EE(|\sum_{n=1}^N\e(h S_n(\vx)|^2)=O_{\alpha, \beta, h}(N).
$$
\end{lemma}
\begin{proof}
First by opening the square we have 
\begin{equation} \label{eq:square}
\begin{aligned}
\left |\sum_{n=1}^N\e(hS_n(\vx))\right |^2 &=\sum_{1\le m, n\le N} \e(hS_n(\vx)-hS_m(\vx))\\
&=N+ \sum_{1\le m<n\le N} \e(hS_n(\vx)-hS_m(\vx)) \\
& \quad \quad \quad +  \sum_{1\le m<n\le N}\overline{\e(hS_n(\vx)-hS_m(\vx))}.
\end{aligned}
\end{equation}

Using the independence of the random variables $X_k$,  we obtain (suppose that $m<n$)
\begin{equation}\label{eq:mn}
\begin{aligned}
\EE(\e(hS_n(\vx)-hS_m(\vx)))&=\EE(\e (\sum_{m<k\le n} h X_k(\vx)))\\
&=\prod_{m<k\le n}  \EE(\e(h X_k(\vx)))=\theta^{n-m},
\end{aligned} 
\end{equation}
where 
$$
\theta=\frac{1}{2}(\e(h\alpha)+\e(h\beta)).
$$

We next proceed on a case by case basis depending on $\e(h\alpha)=\e(h\beta)$ or not.

{\bf Case 1:}  Suppose that 
$\e(h\alpha)\neq \e(h\beta)$. Then $|\theta|<1$. Thus by \eqref{eq:mn} we obtain
\begin{align*}
\sum_{1\le  m<n\le N} &\EE(\e(hS_n(\vx)-hS_m(\vx))) \\
&=\sum_{1\le  m<n\le N} \theta^{n-m} \le \sum_{k=1}^N N \theta^{k}=O(N).
\end{align*} 
Combining with \eqref{eq:square} we obtain the desired bound for this case.

{\bf Case 2:} Now suppose that 
$$
\e(h\alpha)= \e(h\beta),
$$
and $\alpha$ is irrational.  Then for $m<n$,  identity  \eqref{eq:mn} gives, 
$$
\EE(\e(hS_n(\vx)-hS_m(\vx)))=\e(h\alpha(n-m) ).
$$
By \eqref{eq:square} we derive 
\begin{align*}
&\EE(|\sum_{n=1}^N\e(hS_n(\vx)) |^2 )\\
&=N+\sum_{1\le m<n\le N} \e(h\alpha(n-m) )\\
&\quad\quad\quad\quad\quad\quad\quad\quad +
 \sum_{1\le m<n\le N}\overline{\e(h\alpha(n-m))}\\
 &=|\sum_{n=1}^{N}\e(h \alpha n)|^2\ll \frac{1}{|\e(h\alpha)-1|^2}\ll 1.
\end{align*}
The last inequality holds by the fact that $\alpha$ is irrational.  Together with the {\bf Case 1} we finish the proof.
\end{proof}

In fact combining Lemma~\ref{lem:second-moment}  with the criterion  of  Davenport, Erd\H{o}s and  LeVeque~\cite{DEL}, we obtain the result of Theorem~\ref{thm:uniform}. However we show a simper proof for our setting in the following. 

For each $h\in \Z\setminus\{0\}$ let 
\begin{equation}
\label{eq:W}
W_{N, h}(\vx)=\sum_{n=1}^{N}\e(hS_n(\vx)).
\end{equation}
Let $\rho>1$. For  each $i\in \N$ let $N_i=i^{\rho}$. By Lemma~\ref{lem:second-moment}  we obtain 
$$
\EE\left (  (N_i^{-1} W_{N_i, h})^{2}\right )\ll 1/i^{\rho}.
$$
Combining with the monotone convergence  theorem~\cite[Theorem 1.18]{EG}, we obtain
$$
\EE\left (  \sum_{i=1}^{\infty} (N_i^{-1} W_{N_i, h})^{2}\right )=\sum_{i=1}^{\infty} \EE\left (  (N_i^{-1} W_{N_i, h})^{2}\right )<\infty.
$$
Thus for almost all $\vx\in \Sigma$ we have 
$$
\sum_{i=1}^{\infty} (N_i^{-1} W_{N_i, h}(\vx))^{2}<\infty,
$$
and hence 
\begin{equation}
\label{eq:littleO}
W_{N_i, h}(\vx)=o(N_i) \quad \text{as} \quad i\rightarrow \infty.
\end{equation}
We will use this $\vx$  in the following argument.
For each $N$ there exists $N_i$ such that $N_i\le N<N_{i+1}$. Note that 
$$
N_{i+1}-N_i\ll i^{\rho-1}\ll N^{1-1/\rho}.
$$
Clearly we have 
$$
\left| \sum_{N_i<n\le N}\e(h S_n(\vx)) \right| \ll N_{i+1}-N_i \ll N^{1-\rho}.
$$
Thus combining with~\eqref{eq:littleO} we obtain 
$$
W_{N, h}(\vx)= W_{N_i, h}(\vx)+ \sum_{N_i<n\le N}\e(h S_n(\vx))=o(N),
$$
as $N\rightarrow \infty$. Therefore   for any non-zero integer $h$ we have that  for almost all $\vx\in \Sigma$ one has 
$$
W_{N, h}(\vx)=o(N).
$$
Since a countable union of sets with measure   zero has measure zero, we conclude that for almost all $\vx\in \Sigma$  and  for any   non-zero integer $h$  one has 
$$
W_{N, h}(\vx)=o(N).
$$
Hence by using Weyl criterion~\ref{lem:Weyl} we finish the proof of Theorem~\ref{thm:uniform}.
\subsection{Proof of Theorem \ref{thm:exponentialsums}}

We will use the completion trick of \cite[Lemma 3.2]{ChSh} to our setting, we present the proof here for completeness. 

\begin{lemma} 
\label{lem:completion}
Let $h\in \Z\setminus\{0\}$. Then for any $1\le M\le N$ and $\vx\in \Sigma$ we have 
\begin{equation*}
\sum_{n=1}^{M}\e(hS_n(\vx)) \ll U_{N,h} (\vx),
\end{equation*}
where 
$$
U_{N, h}(\vx)=\sum_{j=-N}^{N} \frac{1}{|j|+1}\left |\sum_{n=1}^{N} \e(hS_n(\vx))\e(jn/N)\right |.
$$
\end{lemma}
\begin{proof} 

Observe that by the orthogonality 
$$
\frac{1}{N}\sum_{j=1}^{N}\sum_{k=1}^{M} \e\left(j(n-k)/N\right)=
  \begin{cases}
   1 & n=1, \ldots, M, \\
  0 & \text{otherwise}.
  \end{cases}
$$
We also note that for $1 \le j ,M \le N$ we have
$$
  \sum_{k=1}^{M}\e\left(j k/N\right ) \ll  \frac{N}{\min\{j, N - j+1\}},
$$
see \cite[Equation~(8.6)]{IwKow}.  Recalling the notation $ W_{M, h}(\vx)$ at \eqref{eq:W}, i.e.,  
$$
W_{M, h}(\vx)=\sum_{n=1}^{N}\e(hS_n(\vx)).
$$

It follows that 
\begin{align*}
W_{M, h}(\vx)
&=\sum_{n=1}^{N} \e(hS_n(\vx)) \frac{1}{N}\sum_{j=1}^{N}\sum_{k=1}^{M} \e\left(j(n-k)/N\right)\\
& =\frac{1}{N}\sum_{j=1}^{N} \sum_{k=1}^{M} \e(-jk/N) \sum_{n=1}^{N} \e(hS_n(\vx))\e(jn/N)\\
&\ll \frac{1}{N}\sum_{j=1}^{N} \frac{N}{\min\{j, N - j+1\}}\left |\sum_{n=1}^{N} \e(hS_n(\vx))\e(jn/N)\right |\\
&\ll \sum_{j=-N}^{N} \frac{1}{|j|+1}\left |\sum_{n=1}^{N} \e(hS_n(\vx))\e(jn/N)\right |,
\end{align*}
which completes the proof.
\end{proof}

In analogy of Lemma \ref{lem:second-moment} we obtain 

\begin{lemma} 
\label{lem:U}
Using the notation from Lemma \ref{lem:completion}, 
for each $N\in \N$ we have 
\begin{equation*}
\EE(U_{N, h}(\vx)^{2})\ll N (\log N)^{2}.
\end{equation*}
\end{lemma}
\begin{proof}[Outline of the proof]
Applying  Cauchy-Schwarz inequality for $U_{N, h}(\vx)$, we obtain 
\begin{equation}
\label{eq:U2}
U_{N,h}(\vx)^2\ll \log N \sum_{j=-N}^{N}\frac{1}{|j|+1} \left |\sum_{n=1}^{N} \e(hS_n(\vx))\e(jn/N) \right |^{2}.
\end{equation}
For any integer $j$ with $|j|\le N$ by opening the square and using the same argument as in the proof of Lemma \ref{lem:second-moment},   we obtain 
\begin{equation*}
\EE(|\sum_{n=1}^{N} \e(hS_n(\vx))\e(jn/N) |^{2})\ll N.
\end{equation*}
Together with \eqref{eq:U2} we obtain the desired bound. 
\end{proof}

Now we turn to the proof of Theorem \ref{thm:exponentialsums}.  For each $i\in \N$ let $N_i=2^{i}$. For  $N\in \N$ and $\varepsilon>0$ let
$$
A_{N, h}=\{x\in \Sigma: |U_{N, h}(\vx)|\ge N^{1/2}(\log N)^{3/2+\varepsilon}\}.
$$
By Lemma \ref{lem:U} and Chebyshev inequality we obtain 
\begin{equation*}
\PP(A_{N, h})\ll \frac{1}{(\log N)^{1+2\varepsilon}}.
\end{equation*}
Combining  with the choice $N_i=2^{i}$,   we have 
$$
\sum_{i=1}^{\infty} \PP(A_{{N_i, h}})<\infty.
$$
The Borel-Cantelli lemma implies that for almost all $\vx\in \Sigma$ there exists $i_{\vx}$ such that for any $i\ge i_{\vx}$ we have 
\begin{equation}
\label{eq:control}
U_{{N_i, h}}(\vx)\le N_i^{1/2}(\log N_i)^{3/2+\varepsilon}.
\end{equation}
We now fix this $\vx$ in the following.  
For any $N\ge N_{i_{\vx}}$ there exists $i$ such that 
$$
N_{i}\le N< N_{i+1}.
$$
Applying Lemma \ref{lem:completion} and \eqref{eq:control} we obtain 
\begin{equation*}
W_{N_i, h}(\vx)\ll U_{N_{i+1}, h}(\vx)\ll N^{1/2}(\log N)^{3/2+\varepsilon}.
\end{equation*}

Note that we proved  if $\varepsilon>0$ then for almost all $\vx$ one has 
\begin{equation*}
W_{N, h}(\vx)=N^{1/2}(\log N)^{3/2+\varepsilon}.
\end{equation*}
By taking a sequence $\varepsilon_q \searrow 1/2$ as $q\rightarrow \infty$ and  using the fact that a countable union of sets with measure zero has measure zero, we finish the proof.

\subsection{Proof of Theorem \ref{thm:topology}}

Let $0<\tau<\min\{\alpha, \beta, |\beta-\alpha|\}$. The following lemma shows that there exists an orbit such that the orbit does not intersect $(0, \tau)$ for arbitrary length steps.

For each $\vx\in \Sigma$, interval $I\subseteq (0,1)$ and $n\in \N$ define (hitting times)
\begin{equation*}
h(\vx, I, n)=\#\{1\le j\le n: S_j(\vx)\in I\}.
\end{equation*}

\begin{lemma}
\label{lem:jump}
For any $\vu\in \Sigma_n$ and any $m\in \N$ there exists  a word $\vv=\vv(\vu, m)\in \Sigma_{m}$ such that for each $\vx\in [\vu\vv]$ one has 
\begin{equation*}
S_j(\vx) \notin (0, \tau), \, \forall n<j\le n+m.
\end{equation*}
Moreover for each $\vx\in [\vu\vv]$ we have 
\begin{equation*}
\frac{h(\vx, (0, \tau), n+m)}{n+m}\le \frac{n}{n+m}.
\end{equation*}
\end{lemma}

Now we turn to the proof of Theorem \ref{thm:topology}. For each $\vu\in \Sigma_n$ by Lemma \ref{lem:jump} there exists $\vv=\vv(\vu, n^2)\in \Sigma_{n^2}$ such that for any $\vx\in [\vu\vv]$ it has the properties Lemma \ref{lem:jump}.  

For each $n\in \N$ let 
\begin{equation*}
B_n=\bigcup_{\vu\in \Sigma_n} [\vu \vv(\vu, n^{2})],
\end{equation*}
and 
\begin{equation*}
G=\bigcap_{k=1}^{\infty}\bigcup_{n=k}^{\infty} B_n.
\end{equation*}
Note that for each $k\in \N$ the set  $\bigcup_{n=k}^{\infty} B_n$ is a dense open set, thus  the set $G$ is a dense $\G_\delta$ set. 


Let $\vx\in G$ then there exists a sequence $n_k\nearrow \infty $ as $k\rightarrow \infty$  such that 
$\vx\in B_{n_k}$ for each $k\in \N$.  Moreover for $\vx\in B_{n_k}$ there exists $\vu\in \Sigma_{n_k}$ such that 
$$
\vx\in [\vu\vv(\vu, {n}_k^2)].
$$
Applying Lemma \ref{lem:jump} we obtain
\begin{equation}
\label{eq:break}
\frac{h(\vx, (0, \tau), n_k+n_k^2)}{n_k+n_k^{2}}\le \frac{n_k}{n_k+n_k^{2}}.
\end{equation}
Clearly the right side of \eqref{eq:break} goes to zero as $k\rightarrow \infty$, which implies that  $\vx\in E_{\alpha, \beta}$ and hence finishes the proof.

\subsection{Proof of Theorem \ref{thm:dimension}} 

Let $0<\tau<\min\{\alpha, \beta, |\beta-\alpha|\}$ such that $1/\tau \in \N$.  Divide the interval $[0,1]$ into $q:=1/\tau$ subintervals in a natural way, and denote the collection of these intervals  by
$$
\I=\{I_1, \ldots, I_q\}.
$$
In analogy of Lemma \ref{lem:jump} we formulate the following result.

\begin{lemma}
\label{lem:jumper2}
Let $\varepsilon>0$. For any  $\vu\in \Sigma_n, n\in \N$ with  $n\varepsilon>10$, there exist $\vv \in \Sigma_{\floor{\varepsilon n}}$ and $I\in \I$
such that for any $\vx\in [\vu\vv]$ one has 
\begin{equation}
\label{eq:jumper-i}
S_j(\vx)\notin I, \, \forall n< j\le n+\floor{\varepsilon n},
\end{equation}
and 
\begin{equation*}
\frac{h(\vx, I, n+\floor{\varepsilon n})}{n+\floor{\varepsilon n}}\le \frac{\tau}{1+\varepsilon/2}.
\end{equation*}
\end{lemma}
\begin{proof}
Let $\vu\in \Sigma_n$. For $\vx\in [\vu]$ and the sequence $S_1(\vx), \ldots, S_n(\vx)$ the pigeonhole principle  implies that there exists an interval $I\in \I$
such that 
\begin{equation*}
 h(\vx, I, n) \le n \tau. 
\end{equation*}
For the interval $I$ by using the similar argument as in the proof of Lemma \ref{lem:jump} there exists $\vv=\vv(\vu)\in \Sigma_{\floor{\varepsilon n}}$ such that for any $\vx\in [\vu\vv]$ the orbit $S_j(\vx)$ has property of \eqref{eq:jumper-i}.  It follows that 
\begin{equation*}
\frac{h(\vx, I, n+\floor{\varepsilon n})}{n+\floor{\varepsilon n}}\le \frac{n\tau}{n+\floor{\varepsilon n}},
\end{equation*}
which finishes the proof.
\end{proof}

Now we show the proof of Theorem \ref{thm:dimension}.  Fix $\varepsilon>0$.   For each $\vu\in \Sigma_n$ let  $\vv(\vu)\in \Sigma_{\floor{\varepsilon n}}$ be the same word as in Lemma \ref{lem:jumper2}.   For each $n\in \N$ let
\begin{equation*}
B_n=\bigcup_{\vu\in \Sigma_n} [\vu \vv(\vu)],
\end{equation*}
and 
\begin{equation*}
F=\bigcap_{k=1}^{\infty}\bigcup_{n=k}^{\infty} B_n.
\end{equation*}

Now we show that  $F\subseteq E_{\alpha, \beta}$. Since for $\vx\in F$ there exists a sequence  $n_k\nearrow \infty$ as $ k \rightarrow\infty$ such that $\vx\in B_{n_k}$ for all $ k\in \N$. Applying Lemma \ref{lem:jumper2} there exists an interval $I(\vx, n_k)\in \I$ such that 
\begin{equation}
\label{eq:density-bad}
\frac{h(\vx, I(\vx, n_k), n_k+\floor{\varepsilon n_k})}{n_k+\floor{\varepsilon n_k}}\le \frac{\tau}{1+\varepsilon/2}.
\end{equation}
Since there are $q$ intervals in $\I$, again by pigeonhole principle there exists an interval $I(\vx)$ such that there are infinitely many $k_i, i\in \N$ such that 
\begin{equation*}
 I(\vx)=I(\vx, n_{k_i}), \quad \forall i\in \N.
\end{equation*} 
Together with \eqref{eq:density-bad} we obtain that 
\begin{equation*}
\liminf_{N\rightarrow \infty}\frac{h(\vx, I(\vx), N)}{N} \le \frac{\tau}{1+\varepsilon/2},
\end{equation*}
which implies $\vx\in E_{\alpha, \beta}$.

Now we use some notation from Subsection \ref{sec:Cantor}. Let $n_k\nearrow\infty$ rapidly such that 
\begin{equation*}
\lim_{k\rightarrow \infty}\frac{\log \prod_{j=1}^k q_j}{\log q_{k+1}}=0.
\end{equation*}
For each $k\in \N$ let 
\begin{equation*}
E_k=\bigcap_{j=1}^{k} B_{n_j},
\end{equation*}
and 
\begin{equation*}
E=\bigcap_{k=1}^{\infty} E_k.
\end{equation*}
Applying Corollary \ref{cor:dimension} we obtain 
\begin{equation*}
\dim_H E =1/(1+\varepsilon).
\end{equation*}
Since $E\subseteq F \subset E_{\alpha, \beta}$, we have 
\begin{equation*}
\dim_{H}E_{\alpha, \beta} \ge 1/(1+\varepsilon).
\end{equation*}
By the arbitrary choice of $\varepsilon>0$ and the space $\Sigma$ has Hausdorff dimension one, we obtain $\dim_H E_{\alpha, \beta}=1$ which completes the proof.

\section{Further comments}

In this section we extend Theorem \ref{thm:uniform} to more variables setting. To be precise we show some notation first. In the following let $\alpha_1, \ldots, \alpha_{\ell} \in \T$ with $\ell\ge 2$. For convenience of notation we denote  
$$
\Sigma= \{(x_1, x_2, \ldots, x_n, \ldots): x_n\in \{1, 2, \ldots, \ell\}\}.
$$
For $\vx\in \Sigma$ and each $n\in \N$ define
$$
S_n(\vx)=\sum_{i=1}^{n} X_i(\vx) 
$$
where
\begin{equation*}
X_i(\vx) = \alpha_k \quad \text{provided} \quad x_i=k.
\end{equation*}
The sequence $S_n(\vx), n\in \N$ is called an $(\alpha_1, \ldots, \alpha_\ell)$-orbit. Again we may consider $\Sigma$ as the probability space, and $X_i, i\in \N$ is a sequence of  random variables. For each $i $ we ask that the random variable $X_i$ takes  value from the set $\alpha_1, \ldots, \alpha_{\ell}$ such that  each of them has the equal probability  $1/\ell$ to be chosen, i.e., for each $k=1, \dots, \ell$, one has 
$$
\PP(X_i=\alpha_k)=1/\ell.
$$

In analogy of Theorem \ref{thm:uniform} we have the following. We show the idea only, and omit the similar arguments here.

\begin{theorem}
\label{thm:uniform-more}
Suppose that at least one of $\alpha_k, k=1,\ldots, \ell$ is irrational.   Then  for almost all $\vx\in \Sigma$ the orbit $S_n(\vx)$ is uniformly distributed modulo one.
\end{theorem} 
\begin{proof}
Similarly in the proof of Theorem \ref{thm:uniform}, it is sufficient to show that     for any $h\ne 0$  one has 
\begin{equation*}
\EE(|\sum_{n=1}^N\e(h S_n(\vx)|^2)=O_{\alpha, \beta, h}(N).
\end{equation*}
Applying the similar argument as in the proof of Lemma \ref{lem:second-moment}, for any $m<n$ we have
\begin{equation*}
\EE(\e(h(S_n(\vx)-S_m(\vx)))=\widetilde{\theta}^{n-m},
\end{equation*}
where 
\begin{equation*}
\widetilde{\theta}=\frac{1}{\ell}(\e(h\alpha_1)+\ldots+\e(h\alpha_\ell)).
\end{equation*}
Proceed a case by case basis depending on whether $|\widetilde{\theta}|<1$ or not. 

If $|\widetilde{\theta}|<1$, then the argument in the proof of Lemma \ref{lem:second-moment} gives the desired bound. 

For the case $|\widetilde{\theta}|=1$, we derive that 
$$
\e(h\alpha_1)=\ldots=\e(h\alpha_\ell).
$$
We assume that  $\alpha_1$ is irrational, and we obtain
$$
\widetilde{\theta}=\e(h\alpha_1).
$$
Then  applying the  argument as in the proof of Lemma \ref{lem:second-moment}, we obtain the desired bound for this case also.  Furthermore  applying the similar argument as in the proof of~\ref{thm:uniform} we obtain the desired result.
\end{proof}

Note that the above probability measure on $\Sigma$ is a special  invariant measure on $\Sigma$ (under the shift map). 
It would be interesting to  know whether the Theorem \ref{thm:uniform-more} is still holds for  any invariant measure on $\Sigma$.

\section*{Acknowledgement}

The first and second authors thank Prof. De-jun Feng for hospitality while they were visiting The Chinese University of Hong Kong. They also thank hospitality from Hubei University.  We thank Prof. Igor Shparlinski for helpful comments.

The first author  is  supported in part by ARC Grant~DP170100786. The second author is supported by the NSFC (grants No. 11871200, 11571387).  The third author is supported by the NSFC (grants No. 11871200, 11271114, and 11671189).

\end{document}